\documentclass[12pt,draft]{amsart}

\usepackage{amssymb}

\newtheorem{theorem}{Theorem}[section]

\newtheorem{proposition}[theorem]{Proposition}

\newtheorem{corollary}[theorem]{Corollary}

\newtheorem{lemma}[theorem]{Lemma}

\theoremstyle{definition}

\newtheorem{remark}[theorem]{Remark}

\DeclareMathOperator{\Con}{Con}

\DeclareMathOperator{\Part}{Part}

\DeclareMathOperator{\Sub}{Sub}

\DeclareMathOperator{\var}{var}

\title[identities in the lattice of overcommutative varieties]
{Identities and quasiidentities\\
in the lattice of overcommutative\\
semigroup varieties}

\author{V.\,Yu.\,Shaprynski\v{\i}}

\address{Institute of Mathematics and Computer Science, Ural Federal University,
Lenina 51, 620083 Ekaterinburg, Russia}

\email{vshapr@yandex.ru}

\thanks{The article was partially supported by the Russian Foundation
for Basic Research (grant No. 10-01-00524) and the Ministry of Education
and Science of the Russian Federation (project No. 2.1.1/13995)}



\date{}

\begin{document}

\maketitle

\begin{abstract}
We describe overcommutative varieties of semigroups whose lattice of
overcommutative subvarieties satisfies a non-trivial identity or
quasiidentity. These two properties turn out to be equivalent.
\end{abstract}

\section{Introduction and summary} \label{intr}

It is generally known that the lattice of all semigroup varieties is a
disjoint union of two wide and important sublattices: the ideal of all
periodic varieties and the co-ideal of all \emph{overcommutative} varieties,
that is, varieties containing the variety $\mathcal{COM}$ of all commutative
semigroups. We denote the lattice of all overcommutative varieties by
$\mathbf{OC}$.

By $\mathrm L(\mathcal V)$ we denote the subvariety lattice of a
semigroup variety $\mathcal V$. Identities and quasiidentities in
lattices $\mathrm L(\mathcal V)$ were investigated in several
papers, see Sections~11 and 12 in the
survey~\cite{Shevrin-Vernikov-Volkov}. The results
of~\cite{Burris-Nelson} and \cite{Pudlak-Tuma} imply that no
non-trivial lattice quasiidentity holds in the lattice of
commutative semigroup varieties and hence in the lattice $\mathrm
L(\mathcal V)$ whenever $\mathcal V$ is overcommutative. Therefore
investigation of identities and quasiidentities in lattices $\mathrm
L(\mathcal V)$ gives no information about the lattice $\mathbf{OC}$.
In view of this fact it is natural to study identities and
quasiidentities in lattices of overcommutative subvarieties of
overcommutative varieties. For an overcommutative variety $\mathcal
V$, its lattice of overcommutative subvarieties (that is, the
interval between $\mathcal{COM}$ and $\mathcal V$) will be denoted
by $\mathrm L_{\mathbf{OC}}(\mathcal V)$.

The structure of the lattice $\mathbf{OC}$ has been revealed by
Volkov in~\cite{Volkov}. We shall give the formulations of the
results of this paper in Section~\ref{decomp}. Basing on the results
of~\cite{Volkov}, Vernikov described overcommutative varieties whose
lattice of overcommutative subvarieties is distributive, modular,
arguesian, lower or upper semimodular, lower or upper
semidistributive or satisfies some other related
restrictions~\cite{Vernikov-99},\cite{Vernikov-01}. In the present
paper we describe overcommutative varieties $\mathcal V$ whose
lattice $\mathrm L_{\mathbf{OC}}(\mathcal V)$ satisfies a
non-trivial lattice identity or quasiidentity.

We need the following definitions and notation. Lattices are called
[\emph{quasi}]\emph{equationally} equivalent if they satisfy the
same [quasi]identities. A semigroup variety $\mathcal V$ is
\emph{permutative} if it satisfies an identity of the form
\begin{equation}
\label{permut}x_1x_2\dots x_n=x_{g(1)}x_{g(2)}\dots x_{g(n)}
\end{equation}
where $g$ is a non-trivial permutation on the set $\{1,\dots,n\}$. The
semigroup variety given by an identity system $\Sigma$ is denoted by
$\var\Sigma$. Put
\begin{align*}
&\mathcal{LZ}=\var\{xy=x\},\ \mathcal{RZ}=\var\{xy=y\},\\
&\mathcal X=\var\{xyzt=xytz,\ x^2y^2=y^2x^2=(xy)^2\}\ldotp
\end{align*}
The variety dual to $\mathcal X$ is denoted by
$\overleftarrow{\mathcal X}$.

The main result of this article is

\begin{theorem}
\label{main} For an overcommutative semigroup variety $\mathcal V$,
the following are equivalent:

\begin{itemize}
\item[$a)$]the lattice $\mathrm L_{\mathbf{OC}}(\mathcal V)$ satisfies
a non-trivial lattice identity;
\item[$b)$]the lattice $\mathrm L_{\mathbf{OC}}(\mathcal V)$ satisfies
a non-trivial lattice quasiidentity;
\item[$c)$]the lattice $\mathrm
L_{\mathbf{OC}}(\mathcal V)$ is equationally equivalent to a finite
lattice;
\item[$d)$]the lattice $\mathrm L_{\mathbf{OC}}(\mathcal V)$ is
quasiequationally equivalent to a finite lattice;
\item[$e)$]the variety $\mathcal V$ is permutative and contains none of the varieties
$\mathcal{LZ}$, $\mathcal{RZ}$, $\mathcal X$,
$\overleftarrow{\mathcal X}$.
\end{itemize}

\end{theorem}

Since every finite lattice has a finite identity
basis~\cite{McKenzie}, Theorem~\ref{main} immediately imply the
following

\begin{corollary}
\label{finite basis}If $\mathcal V$ is an overcommutative variety
and the lattice $\mathrm L_{\mathbf{OC}}(\mathcal V)$ satisfies a
non-trivial identity then this lattice has a finite identity
basis.\qed
\end{corollary}

The article consists of four sections. Sections~\ref{decomp} and
\ref{prel} contain preliminary results. In Section~\ref{proof} the
proof of Theorem~\ref{main} is given.

\section{Subdirect decomposition of the lattice $\mathbf{OC}$}
\label{decomp}

The aim of this section is to formulate the results
of~\cite{Volkov}. In order to do this, we need some new definitions
and notation. The free semigroup over the countably infinite
alphabet $X=\{x_1,x_2,\dots\}$ is denoted by $F$. The symbol
$\equiv$ stands for the equality relation on $F$. Put
$X_m=\{x_1,\dots,x_m\}$. Let $F_m$ be the free semigroup over the
set $X_m$. If $w$ is a word then we denote the length of $w$ by
$\ell(w)$ and the number of occurrences of a letter $x_i$ in $w$ by
$\ell_{x_i}(w)$ or, shortly, by $\ell_i(w)$. The symmetric group on
the set $\{1,\dots,m\}$ is denoted by $\mathbb S_m$. For
$g\in\mathbb S_m$ and $1\le i\le m$, we put $g(x_i)=x_{g(i)}$ thus
identifying $\mathbb S_m$ with the symmetric group on $X_m$. The
lattice of all equivalence relations on a set $A$ is denoted by
$\Part(A)$. A set $A$ on which a group $G$ acts is called a
$G$-\emph{set}. A $G$-set can be considered as a unary algebra with
the set $G$ of operations. This observation, in particular, allows
us to consider congruences of $G$-sets. The congruence lattice of a
$G$-set $A$ is denoted by $\Con(A)$. If $L$ is a lattice and $x\in
L$ then $(x]$ (respectively, $[x)$) stands for the principal ideal
(respectively, co-ideal) generated by the element $x$. By $\overline
L$ we denote the dual lattice to a lattice $L$.

A \emph{partition} is a sequence of positive integers
$\lambda=(\lambda_1,\dots,\lambda_m)$ where
$\lambda_1\ge\dots\ge\lambda_m$ and $m\ge 2$. The set of all
partitions is denoted by $\Lambda$. Let us fix a partition
$\lambda$. We say that $\lambda$ is a \emph{partition of the number}
$n$ \emph{into} $m$ \emph{parts} where
$n=\sum\limits_{i=1}^{m}\lambda_i$. The numbers $\lambda_i$ are
called \emph{components} of $\lambda$. We consider the set
$$W_\lambda=\{w\in F_m\mid\ell_i(w)=\lambda_i\ \text{for}\ 1\le i\le m\}$$
and the group
$$G_\lambda=\{g\in\mathbb S_m\mid\lambda_i=\lambda_{g(i)}\ \text{for}\ 1\le i\le m\}\ldotp$$
Every element $g\in G_\lambda$, as a permutation on the alphabet
$X_m$, defines a permutation on the set $W_\lambda$ which renames
letters in each word in $W_\lambda$. This means that the group
$G_\lambda$ acts on the set $W_\lambda$ and this set is considered
as a $G_\lambda$-set. For an overcommutative variety $\mathcal V$,
we define an equivalence relation $\varphi_\lambda(\mathcal V)$ on
$W_\lambda$ as the restriction to the set $W_\lambda$ of the fully
invariant congruence on $F$ corresponding to $\mathcal V$. Thus a
mapping
$\varphi_\lambda\colon\mathbf{OC}\longrightarrow\Part(W_\lambda)$ is
defined.

\begin{proposition}[\cite{Volkov}]
\label{decomposition OC}Every mapping $\varphi_\lambda$ is a
homomorphism of the lattice $\mathbf{OC}$ onto the lattice
$\overline{\Con(W_\lambda)}$. These homomorphisms are components of
an embedding
$$\varphi=(\varphi_\lambda)_{\lambda\in\Lambda}\colon\mathbf{OC}\longrightarrow\prod\limits_{\lambda\in\Lambda}\overline{\Con(W_\lambda)}$$
which decomposes the lattice $\mathbf{OC}$ into a subdirect product
of the lattices $\overline{\Con(W_\lambda)},\
\lambda\in\Lambda$.\qed
\end{proposition}

One can generalize Proposition~\ref{decomposition OC} in order to
obtain a subdirect decomposition of the lattice
$\mathrm{L}_\mathbf{OC}(\mathcal V)$. As a surjective homomorphism,
$\varphi_\lambda$ maps principal ideals to principal ideals, so
$$\varphi_\lambda\bigl(\mathrm{L}_\mathbf{OC}(\mathcal V)\bigr)=
\bigl(\varphi_\lambda(\mathcal V)\bigr]_{\overline{\Con(W_\lambda)}}=
\overline{\bigl[\varphi_\lambda(\mathcal V)\bigr)_{\Con(W_\lambda)}}\ldotp$$
The co-ideal $\bigl[\varphi_\lambda(\mathcal V)\bigr)_{\Con(W_\lambda)}$ is
isomorphic to the congruence lattice of the factor $G_\lambda$-set
$W_\lambda/\varphi_\lambda(\mathcal V)$. Thus we have

\begin{corollary}[\cite{Volkov}]
\label{decomposition general}For any variety $\mathcal
V\in\mathbf{OC}$, the homomorphism
$\varphi|_{\mathrm{L}_\mathbf{OC}(\mathcal V)}$ defines a
decomposition of the lattice $\mathrm{L}_\mathbf{OC}(\mathcal V)$
into a subdirect product of the lattices
$\overline{\Con\bigl(W_\lambda/\varphi_\lambda(\mathcal V)\bigr)}$.\qed
\end{corollary}

Another result we need is

\begin{proposition}[\cite{Volkov}]
\label{embed}Every lattice
$\overline{\Con\bigl(W_\lambda/\varphi_\lambda(\mathcal V)\bigr)}$ can be
embedded into $\mathrm{L}_\mathbf{OC}(\mathcal V)$.\qed
\end{proposition}

\section{Preliminaries on semigroup identities}
\label{prel}

In this section we study some equational properties of the varieties
$\mathcal{LZ}$, $\mathcal{RZ}$, $\mathcal X$ and
$\overleftarrow{\mathcal X}$. The following two lemmas and their duals give
the solution of word problem in these varieties. For the varieties
$\mathcal{LZ}$ and $\mathcal{RZ}$ it is generally known and evident.

\begin{lemma}
\label{LZ word problem}An identity $u=v$ holds in $\mathcal{LZ}$ if
and only if the words $u$ and $v$ start with the same letters.\qed
\end{lemma}

An identity $u=v$ is called \emph{balanced} if $\ell_x(u)=\ell_x(v)$
for every $x\in X$. All identities satisfied by overcommutative
varieties are balanced. A letter $x$ in a word $w\in F$ is called
\emph{simple} if $\ell_x(w)=1$ and \emph{multiple} otherwise.

\begin{lemma}
\label{X word problem}An identity $u=v$ holds in $\mathcal X$ if and
only if it is balanced and at least one of the following holds:

\begin{itemize}
\item[(i)] $u\equiv v\in X$;
\item[(ii)] $u$ and $v$ have equal first letters and equal second letters;
\item[(iii)] $u$ and $v$ have equal first letters and their second letters are multiple;
\item[(iv)] the first and the second letters in $u$ and $v$ are
multiple.
\end{itemize}

\end{lemma}

\begin{proof}
Let us denote by $\alpha$ the fully invariant congruence on $F$
corresponding to $\mathcal X$ and by $\beta$ the set of all balanced
identities $u=v$ (considered as pairs of words) satisfying one of
the conditions (i)--(iv). We must prove that $\alpha=\beta$.

First, one can prove that $\alpha\subseteq\beta$. The identity
$xyzt=xytz$ satisfies (ii) while the identities
$x^2y^2=y^2x^2=(xy)^2$ satisfy (iv), so all these identities belong
to $\beta$. Straightforward verification shows that $\beta$ is a
fully invariant congruence on $F$. This implies the desired
inclusion.

It remains to verify that $\beta\subseteq\alpha$. We shall prove
that a balanced identity $u=v$ holds in $\mathcal X$ in each of the
cases (i)--(iv). The case~(i) is trivial. The identity $xyzt=xytz$
implies every identity of the kind~\eqref{permut} with
$g(1)=1$ and $g(2)=2$. Identifying and renaming letters in the latter
identity, one can obtain every identity with the property (ii). In
the rest of the proof we suppose that the identity $u=v$ is written
in the form $xya=ztb$ where $x,y,z,t\in X$ and $a,b\in F$ (the
letters $x$, $y$, $z$, and $t$ are not assumed to be distinct).
Consider the case~(iv). Suppose that $x\equiv z\equiv t$ and
$x\not\equiv y$, that is $u\equiv xya$ and $v\equiv x^2b$. Since $y$
is multiple, there exist balanced identities of the form
$xya=(xy)^2c$ and $x^2b=x^2y^2c$ for some $c\in F$. These identities
satisfy (ii), so they hold in $\mathcal X$. Hence we have
$$xya=(xy)^2c=x^2y^2c=x^2b$$
in $\mathcal X$. The same arguments show
that $\mathcal X$ satisfies $u=v$ whenever $y\equiv z\equiv t$ and
$x\not\equiv y$ (one should use the identity $(xy)^2=y^2x^2$ rather
than $(xy)^2=x^2y^2$ in this case). Therefore in the general case
$\mathcal X$ satisfies
$$xya=x^2c=xtd=t^2e=ztb\quad\text{where}\ c,d,e\in F$$
whenever these identities are balanced. Of course, such words $c$,
$d$, and $e$ exist, so we are done in the case~(iv). In the case~(iii) the
identity $u=v$ is $xya=xtb$ where $y$ and $t$ are multiple. We may
suppose that the letter $x$ is simple, because otherwise the
property (iv) holds. In particular, $x\not\equiv y$ and $x\not\equiv
z$. The variety $\mathcal X$ satisfies $xya=xy^2c$ and $xtb=xt^2d$
($c,d\in F$) whenever these identities are balanced (the case~(ii)).
Furthermore, $\mathcal X$ satisfies $y^2c=z^2d$ (the case~(iv)), so it
satisfies $xya=xy^2c=xt^2d=xtb$.
\end{proof}

For a non-negative integer $k$, consider the variety $$\mathcal
P_k=\var\{x_1\dots x_kyzt_1\dots t_k=x_1\dots x_kzyt_1\dots
t_k\}\ldotp$$ This variety satisfies every balanced identity of the
form $acb=adb$ where $\ell(a)=\ell(b)=k$.

\begin{lemma}[\cite{Perkins}]
\label{permutative}Every permutative variety is contained in
$\mathcal P_k$ for some $k$.\qed
\end{lemma}

\begin{lemma}
\label{permutative without LZ}Any overcommutative permutative
variety $\mathcal V$ such that $\mathcal{LZ\nsubseteq V}$ satisfies the
identity
\begin{equation}
\label{no LZ} x^ny^nz^n=y^nx^nz^n
\end{equation}
for any sufficiently large $n$.
\end{lemma}
\begin{proof}
Being permutative, the variety $\mathcal V$ is contained in
$\mathcal P_k$ for some $k$ by Lemma~\ref{permutative}.
Lemma~\ref{LZ word problem} and the fact that
$\mathcal{LZ\nsubseteq V}$ imply that the variety $\mathcal V$ satisfies an
identity $xa=yb$ where $x\not\equiv y$. The identity $xa=yb$ is
balanced because $\mathcal V$ is overcommutative. We may suppose
that $a$ and $b$ contain only the letters $x$ and $y$. If this is
not the case then we identify all other letters with $x$. Assume that
$n\ge k+\ell(a)=k+\ell(b)$. We are going to prove that $\mathcal V$
satisfies all identities of the form $cz^n=dz^n$ where $c$ and $d$
contain only the letters $x$ and $y$ and
$\ell_x(c)=\ell_x(d)=\ell_y(c)=\ell_y(d)=n$. This would imply the
statement of the lemma we prove as a partial case. Take the greatest common
prefix $e$ of the words $c$ and $d$. There are words $c'$ and $d'$
with $c\equiv exc'$ and $d\equiv eyd'$. If $\ell(e)\ge k$ then the
identity
$$cz^n\equiv exc'z^n=eyd'z^n\equiv dz^n$$ holds in
$\mathcal V$ because $\mathcal V\subseteq\mathcal P_k$ and $n>k$.
Suppose that $0\le\ell(e)\le k$. To prove that $\mathcal V$
satisfies $cz^n=dz^n$ in this case, we use inverse induction by
$\ell(e)$. As the induction base we take the case $\ell(e)=k$ which
has already been considered. Now we shall prove the statement for
$\ell(e)<k$ assuming that it is proved for greater $\ell(e)$. Put
$p=\ell_x(e)+\ell_x(b)$ and $q=\ell_y(e)+\ell_y(a)$. The inequality
$n\ge k+\ell(a)=k+\ell(b)$ imply $n>p$ and $n>q$. The variety
$\mathcal V$ satisfies
\begin{align*}
cz^n\equiv exc'z^n&=exax^{n-p}y^{n-q}z^n&&
\text{by the induction assumption}\\
&=eybx^{n-p}y^{n-q}z^n&&\text{because } xa=yb\\
&=eyd'z^n\equiv dz^n&&\text{by the induction assumption},
\end{align*}
as was to be proved.
\end{proof}

\begin{lemma}
\label{permutative without LZ and X}Any overcommutative permutative
variety $\mathcal V$ such that $\mathcal{LZ,X\nsubseteq V}$
satisfies the identity
\begin{equation}
\label{no LZ,X}
xtx^{n-1}y^nz^n=yty^{n-1}x^nz^n
\end{equation}
for any sufficiently large $n$.
\end{lemma}
\begin{proof}
By Lemma~\ref{permutative} we have $\mathcal V\subseteq\mathcal P_k$
for some $k$. By Lemma~\ref{permutative without LZ} the variety
$\mathcal V$ satisfies
\begin{equation}
\label{x^my^mz^m=y^mx^mz^m}x^my^mz^m=y^mx^mz^m
\end{equation}
for some $m\ge k$. The
variety $\mathcal V$ satisfies a balanced identity $u=v$ which fails
in $\mathcal X$. According to Lemma~\ref{X word problem}, there are
four possible cases.

\smallskip

\emph{Case} 1. The first letters in $u$ and $v$ coincide, the second
letters are distinct and at least one of the second letters is
simple. Identifying all letters in $u=v$ except this simple letter,
we obtain an identity of the form
\begin{equation}
\label{xyx^p+q-1=x^p+1yx^q-1}xyx^{p+q-1}=x^{p+1}yx^{q-1}
\end{equation}
for some $p$ and $q$. This identity implies
$xyx^{pr+q-1}=x^{pr+1}yx^{q-1}$ for all positive integers $r$, so
$p$ can be replaced by $pr$ in~\eqref{xyx^p+q-1=x^p+1yx^q-1}. This allows us
to suppose that $p\ge k$. Let us take $n$ with $n\ge m+k$ and $n\ge
p+q$. The variety $\mathcal V$ satisfies
\begin{align*}
xtx^{n-1}y^nz^n&=x^{p+1}tx^{n-p-1}y^nz^n&&
\text{by \eqref{xyx^p+q-1=x^p+1yx^q-1}}\\
&=x^my^mz^mtx^{n-m}y^{n-m}z^{n-m}&&
\text{because }\mathcal V\subseteq\mathcal P_k\\
&=y^mx^mz^mtx^{n-m}y^{n-m}z^{n-m}&&\text{by \eqref{x^my^mz^m=y^mx^mz^m}}\\
&=y^{p+1}ty^{n-p-1}x^nz^n&&
\text{because }\mathcal V\subseteq\mathcal P_k\\
&=yty^{n-1}x^nz^n&& \text{by \eqref{xyx^p+q-1=x^p+1yx^q-1}}\ldotp
\end{align*}

\smallskip

\emph{Case} 2. The first letters in $u$ and $v$ are distinct and at
least one of these letters is simple. Identifying all letters in
$u=v$ except this simple letter we obtain $yx^{p+q}=x^pyx^q$ for
some positive $p$ and non-negative $q$. This identity implies
$xyx^{p+q}=x^{p+1}yx^q$, so we return to the Case 1.

\smallskip

\emph{Case} 3. The second letters in $u$ and $v$ coincide and are
simple while the first letters are distinct and multiple. Let us write the
identity $u=v$ in the form $xtu'=ytv'$. We may suppose that $u'$ and
$v'$ contain only the letters $x$ and $y$ because all other letters
can be identified with $x$. Put $p=\ell_x(v')$ and $q=\ell_y(u')$.
Let us take $n$ with $n\ge k+p$, $n\ge q$, $n\ge k+m$, and $n\ge
m+1$. We have that
\begin{align*}
xtx^{n-1}y^nz^n&=xtx^ku'x^{n-k-p}y^{n-q}z^n&&
\text{because }\mathcal V\subseteq\mathcal P_k\\
&=ytx^kv'x^{n-k-p}y^{n-q}z^n&&\text{because }xtu'=ytv'\\
&=ytx^my^mz^mx^{n-m}y^{n-m-1}z^{n-m}&&
\text{because }\mathcal V\subseteq\mathcal P_k\\
&=yty^mx^mz^mx^{n-m}y^{n-m-1}z^{n-m}&&\text{by \eqref{x^my^mz^m=y^mx^mz^m}}\\
&=yty^{n-1}x^nz^n&&\text{because }\mathcal V\subseteq\mathcal P_k
\end{align*}
holds in the variety $\mathcal V$.

\smallskip

\emph{Case} 4. The first letters in $u$ and $v$ are distinct and multiple, the second letters are distinct, and at least 
one of the second letters is simple. Identifying the first letters in the words $u$ and $v$, we return to the Case 1.
\end{proof}

\section{Proof of Theorem~\ref{main}} \label{proof}

The proof follows the scheme $a)\longrightarrow b)\longrightarrow
e)\longrightarrow d)\longrightarrow c)\longrightarrow a)$. The
implications $a)\longrightarrow b)$ and $d)\longrightarrow c)$ are
obvious. The implication $c)\longrightarrow a)$ holds because every
finite lattice satisfies a non-trivial identity
(see~\cite[Lemma~V.3.2]{Gratzer}, for instance). It remains to
verify the implications $b)\longrightarrow e)\longrightarrow d)$.

\smallskip

$b)\longrightarrow e)$ Arguing by contradiction, suppose that the
property $e)$ fails. We shall prove that every finite lattice can be
embedded into one of the lattices
$\Con\bigl(W_\lambda/\varphi_\lambda(\mathcal V)\bigr)$. Hence every finite
lattice can be embedded into $\mathrm L_{\mathbf{OC}}(\mathcal V)$
by Proposition~\ref{embed}. Since every non-trivial lattice
quasiidentity fails in some finite lattice~\cite{Budkin-Gorbunov},
this will give us the contradiction we need. There are three cases
to consider.

\smallskip

\emph{Case} 1. The variety $\mathcal V$ is not permutative. Consider
the partition $\lambda=(\underbrace{1,\dots,1}_{n\text{ times}})$. For this
partition we have $G_\lambda=\mathbb S_n$. The corresponding
$G_\lambda$-set $W_\lambda$ is regular (i.\,e., it is transitive and
any non-unit element of $G_\lambda$ has no fixed points). In this
case $\Con\bigl(W_\lambda)\cong\Sub(G_\lambda\bigr)=\Sub(\mathbb S_n)$ where
$\Sub(G)$ is the subgroup lattice of a group $G$
(see~\cite[Lemma~4.20]{McKenzie-McNulty-Taylor}). Since the variety
$\mathcal V$ is not permutative, the congruence
$\varphi_\lambda(\mathcal V)$ is the equality relation on
$W_\lambda$, so $W_\lambda/\varphi_\lambda(\mathcal V)=W_\lambda$.
We have obtained that
$\Con\bigl(W_\lambda/\varphi_\lambda(\mathcal V)\bigr)\cong
\Sub(\mathbb S_n)$. Every finite lattice can be embedded
into a lattice $\Sub(\mathbb S_n)$ for some $n$ \cite{Pudlak-Tuma},
so we are done.

\smallskip

\emph{Case} 2. The variety $\mathcal V$ contains one of the
subvarieties $\mathcal{LZ}$ and $\mathcal{RZ}$. By duality
principle, we may suppose that $\mathcal{LZ\subseteq V}$. Consider
the partition $\lambda=(m,m-1,\dots,2,1)$ for an arbitrary $m\ge 2$.
The group $G_\lambda$ is trivial, whence
$\Con\bigl(W_\lambda/\varphi_\lambda(\mathcal V)\bigr)=
\Part\bigl(W_\lambda/\varphi_\lambda(\mathcal V)\bigr)$. Since
$\mathcal{LZ\subseteq V}$, the variety $\mathcal V$ satisfies no
identity $u=v$ where the first letters in $u$ and $v$ are distinct.
In particular, $(x_ia,x_jb)\not\in\varphi_\lambda(\mathcal V)$
whenever $x_ia,x_jb\in W_\lambda$ and $i\neq j$. Hence the set
$W_\lambda/\varphi_\lambda(\mathcal V)$ contains at least $m$
elements. Any finite lattice can be embedded into any sufficiently
large finite partition lattice \cite{Pudlak-Tuma}, so it can be
embedded into some of the lattices
$\Con\bigl(W_\lambda/\varphi_\lambda(\mathcal V)\bigr)$.

\smallskip

\emph{Case} 3. The variety $\mathcal V$ contains one of the
subvarieties $\mathcal X$ and $\overleftarrow{\mathcal X}$, say,
$\mathcal{X\subseteq V}$. Consider the same partition $\lambda$ as
in Case~2. Again we have
$\Con\bigl(W_\lambda/\varphi_\lambda(\mathcal V)\bigr)=
\Part\bigl(W_\lambda/\varphi_\lambda(\mathcal V)\bigr)$. Since
$\mathcal{X\subseteq V}$, Lemma~\ref{X word problem} implies that
the variety $\mathcal V$ satisfies no identity $u=v$ where the first
letters in $u$ and $v$ are distinct and the second letter in $u$ is
simple. In particular,
$(x_ix_ma,x_jx_mb)\not\in\varphi_\lambda(\mathcal V)$ whenever
$x_ix_ma,x_jx_mb\in W_\lambda$ and $i\neq j$. Hence the set
$W_\lambda/\varphi_\lambda(\mathcal V)$ contains at least $m-1$
elements, so we are done, as in Case~2.

\smallskip

$e)\longrightarrow d)$. Let $\mathcal V$ be an overcommutative variety
satisfying $e)$. Consider the subdirect decomposition
of the lattice $\mathrm L_{\mathbf{OC}}(\mathcal V)$ given by
Corollary~\ref{decomposition general}. We will prove that the
cardinalities of the subdirect multipliers
$\overline{\Con\bigl(W_\lambda/\varphi_\lambda(\mathcal V)\bigr)}$ are
bounded. This implies that there exist only finite number of
non-isomorphic lattices among these multipliers. The lattice
$\mathrm L_{\mathbf{OC}}(\mathcal V)$ is quasiequationally
equivalent to the direct product of these distinct multipliers
because quasiidentities are preserved under taking sublattices and
direct products. Therefore the implication will be proved.

Let us fix a partition $\lambda$. The variety $\mathcal V$ is
contained in $\mathcal P_k$ for some $k$ by Lemma~\ref{permutative}.
We may assume that $\lambda$ is a partition of a number greater than
$2k+1$. Indeed, there is only a finite number of other partitions
and existence of an upper bound for
$\bigl|\Con\bigl(W_\lambda/\varphi_\lambda(\mathcal V)\bigr)\bigr|$ does not
depend on them. By Lemmas~\ref{permutative without LZ}, \ref{permutative
without LZ and X} and their duals the variety $\mathcal V$ satisfies the
identities \eqref{no LZ}, \eqref{no LZ,X}, and their duals for some
$n$. Consider the set $I$ of all integers $i$, $1\le i<n+2k$, such
that at least $4k$ components of $\lambda$ are equal to $i$. For
every $i\in I$, we fix a set of letters $Y_i$ such that $|Y_i|=4k$
and $\lambda_j=i$ whenever $x_j\in Y_i$. This means that
$\ell_x(w)=i$ whenever $x\in Y_i$ and $w\in W_\lambda$. Each word
$w\in W_\lambda$ can be written as $w\equiv abc$ where
$\ell(a)=\ell(c)=k$. We denote $a$, $b$, and $c$ by $L(w)$, $M(w)$,
and $R(w)$, respectively. Note that
$(w_1,w_2)\in\varphi_\lambda(\mathcal V)$ whenever $w_1,w_2\in
W_\lambda$, $L(w_1)=L(w_2)$, and $R(w_1)=R(w_2)$. Consider the
following two restrictions on a word $w\in W_\lambda$:
\begin{itemize}
\item[(i)]there are no letters $x$ in the words $L(w)$ and $R(w)$ with $\ell_x(w)\ge n+2k$
and $x\not\equiv x_1$, $x\not\equiv x_2$ (recall that
$\ell_1(w)\ge\ell_2(w)\ge\ell_x(w)$ for any $x\in
X\setminus\{x_1,x_2\}$, so this property trivially holds whenever
$\ell_2(w)<n+2k$);
\item[(ii)]there are no letters $x$ in the words $L(w)$ and $R(w)$ with
$\ell_x(w)=i\in I$ and $x\not\in Y_i$.
\end{itemize}
Let us prove that, for any $w\in W_\lambda$, there exist $w'\in
W_\lambda$ with the property (i) and such that $w=w'$ in $\mathcal
V$. This means that each $\varphi_\lambda(\mathcal V)$-class
contains a word with the property (i). Consider an occurrence in
$L(w)$ of a letter $x$ with $\ell_x(w)>n+2k$, $x\not\equiv x_1$, and
$x\not\equiv x_2$. There are words $d$ and $e$ with $L(w)\equiv
dxe$. Since $\ell_1(w)\ge\ell_2(w)\ge\ell_x(w)\ge n+2k$, we have
$$\ell_x(M(w)),\ell_1(M(w)),\ell_2(M(w))\ge n\ldotp$$ Hence there exists a
balanced identity of the form $M(w)=x^{n-1}x_1^nx_2^nf$ for some
word $f$. The variety $\mathcal V$ satisfies
\begin{align*}
w&\equiv L(w)M(w)R(w)\\
&\equiv dxeM(w)R(w)\\
&=dxex^{n-1}x_1^nx_2^nfR(w)&&
\text{because }\mathcal V\subseteq\mathcal P_k\\
&=dx_1e x_1^{n-1}x^nx_2^nfR(w)&&
\text{by \eqref{no LZ} if }e\text{ is empty}\\
&&&\text{or by \eqref{no LZ,X} otherwise}\ldotp
\end{align*}
The word $w''\equiv dx_1e
x_1^{n-1}x^nx_2^nfR(w)$ is such that $L(w'')\equiv dx_1e$,
$R(w'')\equiv R(w)$, and $(w,w'')\in\varphi_\lambda(\mathcal V)$. We
have excluded one occurrence of the letter $x$ in $L(w)$. Repeating
this procedure one can exclude all occurrences in $L(w)$ of letters
$x$ with $\ell_x(w)>n+2k$ except $x_1$ and $x_2$. Dually, one can
exclude all occurrences of such letters in $R(w)$.

Now we shall prove that every identity $u=v$ such that $u,v\in
W_\lambda$ is equivalent to an identity $u'=v'$ where $u'$ and $v'$
satisfy (ii). Since
$$\ell(L(u))+\ell(L(v))+\ell(R(u))+\ell(R(v))=4k,$$
the words $L(u)$, $L(v)$, $R(u)$, and $R(v)$ contain at most $4k$ distinct
letters. Therefore, for $1\le i<n+2k$, they contain at most $4k$ distinct
letters $x$ with $\ell_x(u)=i$. Consider any element $g\in
G_\lambda$ which maps, for every $1\le i<n+2k$, all letters $x$ in
$L(u),L(v),R(u),R(v)$ with $\ell_x(w)=i$ to the set $Y_i$. To obtain
the identity $u'=v'$, one may take $u'\equiv g(u)$ and $v'\equiv
g(v)$.

Combining the statements in the previous two paragraphs, we conclude
that every identity $u=v$ with $u,v\in W_\lambda$ is equivalent
within the variety $\mathcal V$ to an identity $u'=v'$ where $u'$
and $v'$ satisfy (i) and (ii). This statement may be reformulated in
terms of $G$-sets. To do this, denote by $A$ the set of
$\varphi_\lambda(\mathcal V)$-classes of all words in $W_\lambda$
satisfying (i) and (ii). We have proved that every congruence on
$W_\lambda/\varphi_\lambda(\mathcal V)$ is generated by some subset
of $A\times A$. The $\varphi_\lambda(\mathcal V)$-class of $w$ is
defined by $L(w)$ and $R(w)$ and does not depend on $M(w)$.
Conditions (i) and (ii) mean that $L(w)$ and $R(w)$ for all such $w$
may contain at most $4k(2n+k-1)+2$ distinct letters in common: at
most $4k$ letters $x$ with $\ell_x(w)=i$ for every $1\le i<2n+k$ and
at most 2 letters $x$ with $\ell_x(w)\ge 2n+k$. Hence $|A|\le N$
where $N=\bigl(4k(2n+k-1)+2\bigr)^{2k}$. Therefore
$$\bigl|\Con\bigl(W_\lambda/\varphi_\lambda(\mathcal V)\bigr)\bigr|\le
2^{|A\times A|}\le2^{N^2}\ldotp$$
This upper bound does not depend on the partition
$\lambda$.

Theorem~\ref{main} is proved.\qed

\begin{remark}
\label{cardinalities}The proof of the implication $e)\longrightarrow d)$
bases on the fact that the cardinalities of the lattices
$\Con\bigl(W_\lambda/\varphi_\lambda(\mathcal V)\bigr)$ are bounded whenever
$\mathcal V$ satisfies $e)$. However the cardinalities of the sets
$W_\lambda/\varphi_\lambda(\mathcal V)$ can be unbounded. For
example, put
$$\mathcal V=\var\{x^2y=yx^2,\ xyz=xzy\}\ldotp$$
For the partition $\lambda=(\underbrace{1,\dots,1}_{n\text{ times}})$,
it is easy to verify that the set $W_\lambda/\varphi_\lambda(\mathcal V)$
contains exactly $n$ elements.
\end{remark}

\begin{remark}
\label{COM+LZ}The variety $\mathcal{LZ}$ is generally known to be an atom of
the lattice of all semigroup varieties. Consequently it would be possible to
conjecture that the lattice $\mathrm L_{\mathbf{OC}}(\mathcal{COM\vee LZ})$
is small in a sense. Surprisingly, this conjecture is very far from the real
situation. The proof of Theorem~\ref{main} shows that this lattice contains
an isomorphic copy of every finite lattice (see Case~2 in the proof of the
implication~$b)\longrightarrow e)$).
\end{remark}

\subsection*{Acknowledgemets} The author is grateful to
professor B.\,Vernikov for helpful discussions.

\end{document}